\documentclass[11pt]{amsart}
\usepackage{amsmath, amssymb, amscd, mathrsfs, url, pinlabel,verbatim}
\usepackage[pagebackref]{hyperref}
\usepackage[margin=1in,marginparwidth=0.75in,centering,letterpaper,dvips]{geometry}
\usepackage{color,dcpic,latexsym,graphicx,epstopdf,comment}
\usepackage[all]{xy}
\usepackage[dvipsnames]{xcolor}
\usepackage{tikz,tikz-cd,pgfplots}
\usepackage{upquote,tabularx,textcomp}
\usepackage[shortlabels]{enumitem}
\usepackage{mathtools}
\usepackage[color=blue!20!white,textsize=tiny]{todonotes}

\title{Zero-surgery characterizes infinitely many knots}

\author[John A. Baldwin]{John A. Baldwin}
\address{Department of Mathematics \\ Boston College}
\email{john.baldwin@bc.edu}

\author[Steven Sivek]{Steven Sivek}
\address{Department of Mathematics, Imperial College London}
\email{s.sivek@imperial.ac.uk}

\makeatletter
\newtheorem*{rep@theorem}{\rep@title}
\newcommand{\newreptheorem}[2]{%
\newenvironment{rep#1}[1]{%
 \def\rep@title{#2 \ref{##1}}%
 \begin{rep@theorem}}%
 {\end{rep@theorem}}}
\makeatother

\newtheorem {theorem}{Theorem}
\newreptheorem{theorem}{Theorem}
\newtheorem {lemma}[theorem]{Lemma}
\newtheorem {proposition}[theorem]{Proposition}

\numberwithin{equation}{section}
\numberwithin{theorem}{section}

\theoremstyle{definition}

\newtheorem{remark}[theorem]{Remark}
\newtheorem*{remark*}{Remark}

\newlist{pcases}{enumerate}{1}
\setlist[pcases]{
  label=\bf{Case~\arabic*:}\protect\thiscase.~,
  ref=\arabic*,
  align=left,
  labelsep=0pt,
  leftmargin=0pt,
  labelwidth=0pt,
  parsep=0pt
}
\newcommand{\case}[1][]{%
  \if\relax\detokenize{#1}\relax
    \def\thiscase{}%
  \else
    \def\thiscase{~#1}%
  \fi
  \item
}

\newcommand{\Z}{\mathbb{Z}}

\newcommand{\C}{\mathbb{C}}

\newcommand{\F}{\mathbb{F}}
\newcommand{\Q}{\mathbb{Q}}
\newcommand{\spc}{\operatorname{Spin}^c}
\newcommand{\spinc}{\mathfrak{s}}

\newcommand\hfk{\mathit{HFK}}
\newcommand\cfk{\mathit{CFK}}
\newcommand\cfkinfty{\cfk^\infty}
\newcommand\hfkhat{\widehat{\hfk}}

\DeclareFontFamily{U}{mathx}{\hyphenchar\font45}
\DeclareFontShape{U}{mathx}{m}{n}{
      <5> <6> <7> <8> <9> <10>
      <10.95> <12> <14.4> <17.28> <20.74> <24.88>
      mathx10
      }{}
\DeclareSymbolFont{mathx}{U}{mathx}{m}{n}
\DeclareFontSubstitution{U}{mathx}{m}{n}
\DeclareMathAccent{\widecheck}{0}{mathx}{"71}

\newcommand{\hfhat}{\widehat{\mathit{HF}}}
\newcommand{\hfp}{\mathit{HF}^+}
\newcommand{\hfred}{\mathit{HF}^+_{\mathrm{red}}}

\newcommand{\mirror}[1]{\overline{#1}}

\newcommand{\Wh}{\mathrm{Wh}}

\makeatletter
\DeclareFontFamily{OMX}{MnSymbolE}{}
\DeclareSymbolFont{MnLargeSymbols}{OMX}{MnSymbolE}{m}{n}
\SetSymbolFont{MnLargeSymbols}{bold}{OMX}{MnSymbolE}{b}{n}
\DeclareFontShape{OMX}{MnSymbolE}{m}{n}{
    <-6>  MnSymbolE5
   <6-7>  MnSymbolE6
   <7-8>  MnSymbolE7
   <8-9>  MnSymbolE8
   <9-10> MnSymbolE9
  <10-12> MnSymbolE10
  <12->   MnSymbolE12
}{}
\DeclareFontShape{OMX}{MnSymbolE}{b}{n}{
    <-6>  MnSymbolE-Bold5
   <6-7>  MnSymbolE-Bold6
   <7-8>  MnSymbolE-Bold7
   <8-9>  MnSymbolE-Bold8
   <9-10> MnSymbolE-Bold9
  <10-12> MnSymbolE-Bold10
  <12->   MnSymbolE-Bold12
}{}

\let\llangle\@undefined
\let\rrangle\@undefined
\DeclareMathDelimiter{\llangle}{\mathopen}%
                     {MnLargeSymbols}{'164}{MnLargeSymbols}{'164}
\DeclareMathDelimiter{\rrangle}{\mathclose}%
                     {MnLargeSymbols}{'171}{MnLargeSymbols}{'171}
\makeatother

\newcounter{desccount}

\newcommand{\descref}[1]{\hyperref[#1]{#1}}

\usetikzlibrary{calc,intersections}
\tikzset{every picture/.style=thick}
\tikzset{link/.style = { white, double = black, line width = 1.75pt, double distance = 1.25pt, looseness=1.75 }}
\tikzset{crossing/.style = {draw, circle, dotted, minimum size=0.5cm, inner sep=0, outer sep=0}}
\pgfplotsset{compat=1.12}

\begin{document}

\begin{abstract}
We prove that $0$ is a characterizing slope for infinitely many knots, namely the genus-1 knots whose knot Floer homology is $2$-dimensional in the top Alexander grading, which we classified in recent work and which include all $(-3,3,2n+1)$ pretzel knots.  This was previously only known for $5_2$ and its mirror, as a corollary of that classification, and for the unknot, trefoils, and the figure eight by work of Gabai from 1987.
\end{abstract}

\maketitle

\section{Introduction} \label{sec:intro}

A rational number $r\in\Q$ is said to be a \emph{characterizing slope} for a knot $K \subset S^3$ if the orientation-preserving homeomorphism type of the manifold obtained via Dehn surgery on $K$ of slope $r$ uniquely determines $K$; that is, \[ \text{if } S^3_r(J) \cong S^3_r(K) \text{ then } J=K.\]
It seems very hard to prove for most knots that any given integral slope is characterizing.  This is especially true for slope 0: in his celebrated 1987 work \cite{gabai-foliations3}, Gabai proved  that $S^3_0(K)$ detects the genus of $K$ and whether or not $K$ is fibered, which immediately implies that $0$-surgery characterizes the unknot (resolving the Property R Conjecture), trefoils, and figure eight.  To our knowledge, the only other knots known to be characterized by their $0$-surgeries are $5_2$ and its mirror, which we proved in our recent work \cite{bs-characterizing}. The main result of this paper is that  infinitely many knots are characterized by their $0$-surgeries:

\begin{theorem} \label{thm:main}
Let $K$ be any of the knots \[15n_{43522},\,\, \Wh^-(T_{2,3},2),\,\,
\Wh^+(T_{2,3},2), \,\,P(-3,3,2n+1) \ (n\in\Z),\] or their mirrors. Then $0$ is a characterizing slope for $K$.
\end{theorem}

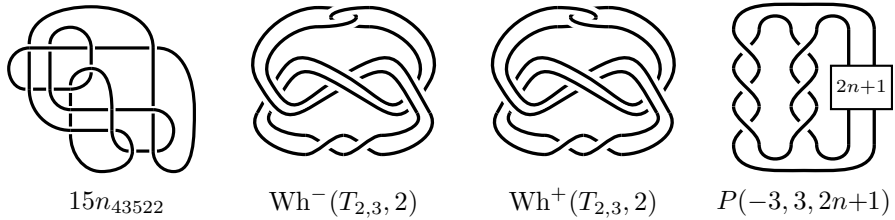
\begin{figure}
\begin{tikzpicture}[scale=0.85,font=\small]
\begin{scope}
\begin{scope}[scale=0.3,very thick] 
    \draw (5.78, 2.76) .. controls (6.31, 2.76) and (6.66, 2.25) .. (6.66, 1.68);
    \draw (6.66, 1.68) .. controls (6.66, 0.97) and (5.86, 0.61) .. 
          (5.05, 0.61) .. controls (4.09, 0.61) and (3.43, 1.54) .. (3.43, 2.56);
    \draw (3.43, 2.96) .. controls (3.43, 3.19) and (3.43, 3.41) .. (3.43, 3.64);
    \draw (3.43, 4.03) .. controls (3.43, 4.26) and (3.43, 4.49) .. (3.43, 4.71);
    \draw (3.43, 5.11) .. controls (3.43, 5.64) and (3.95, 5.99) .. (4.51, 5.99);
    \draw (4.51, 5.99) .. controls (5.39, 5.99) and (5.59, 4.87) .. (5.59, 3.84);
    \draw (5.59, 3.84) .. controls (5.59, 3.48) and (5.59, 3.12) .. (5.59, 2.76);
    \draw (5.59, 2.76) .. controls (5.59, 2.20) and (5.94, 1.68) .. (6.46, 1.68);
    \draw (6.86, 1.68) .. controls (7.09, 1.68) and (7.31, 1.68) .. (7.54, 1.68);
    \draw (7.94, 1.68) .. controls (8.46, 1.68) and (8.81, 2.20) .. 
          (8.81, 2.76) .. controls (8.81, 3.35) and (8.33, 3.84) .. (7.74, 3.84);
    \draw (7.74, 3.84) .. controls (7.09, 3.84) and (6.43, 3.84) .. (5.78, 3.84);
    \draw (5.39, 3.84) .. controls (4.74, 3.84) and (4.08, 3.84) .. (3.43, 3.84);
    \draw (3.43, 3.84) .. controls (2.84, 3.84) and (2.36, 4.32) .. (2.36, 4.91);
    \draw (2.36, 4.91) .. controls (2.36, 5.56) and (2.36, 6.21) .. (2.36, 6.87);
    \draw (2.36, 7.26) .. controls (2.36, 7.79) and (2.87, 8.14) .. 
          (3.43, 8.14) .. controls (4.03, 8.14) and (4.51, 7.66) .. (4.51, 7.06);
    \draw (4.51, 7.06) .. controls (4.51, 6.77) and (4.51, 6.48) .. (4.51, 6.19);
    \draw (4.51, 5.79) .. controls (4.51, 5.26) and (3.99, 4.91) .. (3.43, 4.91);
    \draw (3.43, 4.91) .. controls (3.14, 4.91) and (2.85, 4.91) .. (2.56, 4.91);
    \draw (2.16, 4.91) .. controls (1.93, 4.91) and (1.71, 4.91) .. (1.48, 4.91);
    \draw (1.08, 4.91) .. controls (0.56, 4.91) and (0.21, 5.43) .. 
          (0.21, 5.99) .. controls (0.21, 6.58) and (0.69, 7.06) .. (1.28, 7.06);
    \draw (1.28, 7.06) .. controls (1.64, 7.06) and (2.00, 7.06) .. (2.36, 7.06);
    \draw (2.36, 7.06) .. controls (3.01, 7.06) and (3.66, 7.06) .. (4.31, 7.06);
    \draw (4.71, 7.06) .. controls (5.65, 7.06) and (6.60, 7.06) .. (7.54, 7.06);
    \draw (7.94, 7.06) .. controls (9.32, 7.06) and (9.89, 5.43) .. 
          (9.89, 3.84) .. controls (9.89, 2.37) and (9.89, 0.61) .. 
          (8.81, 0.61) .. controls (8.22, 0.61) and (7.74, 1.09) .. (7.74, 1.68);
    \draw (7.74, 1.68) .. controls (7.74, 2.34) and (7.74, 2.99) .. (7.74, 3.64);
    \draw (7.74, 4.03) .. controls (7.74, 5.04) and (7.74, 6.05) .. (7.74, 7.06);
    \draw (7.74, 7.06) .. controls (7.74, 8.50) and (6.13, 9.21) .. 
          (4.51, 9.21) .. controls (2.92, 9.21) and (1.28, 8.65) .. (1.28, 7.26);
    \draw (1.28, 6.87) .. controls (1.28, 6.21) and (1.28, 5.56) .. (1.28, 4.91);
    \draw (1.28, 4.91) .. controls (1.28, 3.72) and (2.25, 2.76) .. (3.43, 2.76);
    \draw (3.43, 2.76) .. controls (4.08, 2.76) and (4.74, 2.76) .. (5.39, 2.76);
\end{scope}
\node at (1.75,-0.3) {$15n_{43522}$};
\end{scope}

\begin{scope}[xshift=5.3cm,yshift=2.6cm] 
\draw[link] (-0.4,0.15) arc (90:0:0.6 and 0.15) ++ (-0.4,0) arc (180:270:0.6 and 0.15); 
\draw[link] (0.4,0.15) arc (90:180:0.6 and 0.15) ++ (0.4,0) arc (360:270:0.6 and 0.15);
\draw[link] (-1.35,-1.5) to[out=60,in=150] (0,-0.9) to[out=330,in=240] (1.15,-1) to[out=60,in=0,looseness=1] (0.4,-0.15); 
\draw[link] (-1.15,-1.5) to[out=60,in=150] (0,-1.1) to[out=330,in=240] (1.35,-1) to[out=60,in=0,looseness=1] (0.4,0.15);
\draw[link] (-0.4,0.15) to[out=180,in=120,looseness=1] (-1.35,-1) to[out=300,in=210] (0,-1.1) to[out=30,in=120] (1.15,-1.5); 
\draw[link] (-0.4,-0.15) to[out=180,in=120,looseness=1] (-1.15,-1) to[out=300,in=210] (0,-0.9) to[out=30,in=120] (1.35,-1.5);
\begin{scope} 
\clip (-0.5,-1.5) rectangle (0.5,-0.5);
\draw[link] (-1.35,-1.5) to[out=60,in=150] (0,-0.9) to[out=330,in=240] (1.15,-1);
\draw[link] (-1.15,-1.5) to[out=60,in=150] (0,-1.1) to[out=330,in=240] (1.35,-1);
\end{scope}
\draw[link,looseness=1] (-1.35,-1.5) to[out=240,in=180] (-0.6,-2.15) ++ (1.2,0) to[out=0,in=300] (1.35,-1.5);
\draw[link,looseness=1] (-1.15,-1.5) to[out=240,in=180] (-0.6,-1.85) ++ (1.2,0) to[out=0,in=300] (1.15,-1.5);
\draw[link,looseness=0.75] (-0.6,-1.85) to[out=0,in=180] ++(0.6,-0.3) ++(0,0.3) to[out=0,in=180] ++(0.6,-0.3); 
\draw[link,looseness=0.75] (-0.6,-2.15) to[out=0,in=180] ++(0.6,0.3) ++(0,-0.3) to[out=0,in=180] ++(0.6,0.3);
\node at (0,-2.9) {$\Wh^-(T_{2,3},2)$};
\end{scope}

\begin{scope}[xshift=9.05cm,yshift=2.6cm] 
\draw[link] (0.4,0.15) arc (90:180:0.6 and 0.15) ++ (0.4,0) arc (360:270:0.6 and 0.15); 
\draw[link] (-0.4,0.15) arc (90:0:0.6 and 0.15) ++ (-0.4,0) arc (180:270:0.6 and 0.15);
\draw[link] (-1.35,-1.5) to[out=60,in=150] (0,-0.9) to[out=330,in=240] (1.15,-1) to[out=60,in=0,looseness=1] (0.4,-0.15); 
\draw[link] (-1.15,-1.5) to[out=60,in=150] (0,-1.1) to[out=330,in=240] (1.35,-1) to[out=60,in=0,looseness=1] (0.4,0.15);
\draw[link] (-0.4,0.15) to[out=180,in=120,looseness=1] (-1.35,-1) to[out=300,in=210] (0,-1.1) to[out=30,in=120] (1.15,-1.5); 
\draw[link] (-0.4,-0.15) to[out=180,in=120,looseness=1] (-1.15,-1) to[out=300,in=210] (0,-0.9) to[out=30,in=120] (1.35,-1.5);
\begin{scope} 
\clip (-0.5,-1.5) rectangle (0.5,-0.5);
\draw[link] (-1.35,-1.5) to[out=60,in=150] (0,-0.9) to[out=330,in=240] (1.15,-1);
\draw[link] (-1.15,-1.5) to[out=60,in=150] (0,-1.1) to[out=330,in=240] (1.35,-1);
\end{scope}
\draw[link,looseness=1] (-1.35,-1.5) to[out=240,in=180] (-0.6,-2.15) ++ (1.2,0) to[out=0,in=300] (1.35,-1.5);
\draw[link,looseness=1] (-1.15,-1.5) to[out=240,in=180] (-0.6,-1.85) ++ (1.2,0) to[out=0,in=300] (1.15,-1.5);
\draw[link,looseness=0.75] (-0.6,-1.85) to[out=0,in=180] ++(0.6,-0.3) ++(0,0.3) to[out=0,in=180] ++(0.6,-0.3); 
\draw[link,looseness=0.75] (-0.6,-2.15) to[out=0,in=180] ++(0.6,0.3) ++(0,-0.3) to[out=0,in=180] ++(0.6,0.3);
\node at (0,-2.9) {$\Wh^+(T_{2,3},2)$};
\end{scope}

\begin{scope}[xshift=12.5cm,yshift=2.4cm]
\draw[link,looseness=1] (-1.1,0) to[out=90,in=180] ++(1.1,0.4) to[out=0,in=90] ++(1.1,-0.4);
\draw[link,looseness=1.5] (-0.7,0) to[out=90,in=90] ++(0.5,0) ++(0.4,0) to[out=90,in=90] ++(0.5,0);
\foreach \i in {0,-0.6,-1.2} {
  \draw[link,looseness=0.75] (-0.7,\i) to[out=270,in=90] ++(-0.4,-0.6) (-0.2,\i) to[out=270,in=90] ++(0.4,-0.6);
  \draw[link,looseness=0.75] (-1.1,\i) to[out=270,in=90] ++(0.4,-0.6) (0.2,\i) to[out=270,in=90] ++(-0.4,-0.6);
}
\draw[link] (0.7,0) -- ++(0,-1.8) ++(0.4,0) -- ++(0,1.8);
\node[draw,thick,rectangle,fill=white,inner sep=2pt,minimum height=1.5em] at (0.9,-0.9) {\tiny$2n{+}1$};
\draw[link,looseness=1.5] (-0.7,-1.8) to[out=270,in=270] ++(0.5,0) ++(0.4,0) to[out=270,in=270] ++(0.5,0);
\draw[link,looseness=1] (-1.1,-1.8) to[out=270,in=180] ++(1.1,-0.4) to[out=0,in=270] ++(1.1,0.4);
\node at (0,-2.7) {$P(-3,3,2n{+}1)$};
\end{scope}

\end{tikzpicture}
\caption{The knots that Theorem~\ref{thm:main} says are characterized by their $0$-surgeries.}
\label{fig:main-knots}
\end{figure}

Here, $\Wh^\pm(T_{2,3},2)$ is the $2$-twisted Whitehead double of the right-handed trefoil, with a positive or a negative clasp, respectively, and the $P(-3,3,2n+1)$ are pretzel knots.  See Figure~\ref{fig:main-knots}.

By contrast, there are many knots that are not characterized by their $0$-surgeries.  Brakes \cite{brakes} gave the first pairs of examples, and later Osoinach \cite{osoinach} used annulus twisting to construct infinite families of examples.  In fact, there can be  infinitely many knots $K_n$ with pairwise diffeomorphic $0$-traces $X_0(K_n)$, the result of attaching a 0-framed 2-handle to $B^4$ along $K_n$ \cite{AJOT}.  Knots which are not smoothly concordant, or which have different slice genera, can nonetheless have diffeomorphic $0$-surgeries \cite{yasui-corks} or even $0$-traces \cite{miller-piccirillo,piccirillo-shake}. Indeed, Piccirillo \cite{piccirillo-conway} famously proved that the Conway knot is not slice by exhibiting a non-slice knot with the same $0$-trace.  Recently, Manolescu and Piccirillo \cite{manolescu-piccirillo} have given a systematic construction of pairs of knots with the same $0$-surgeries, and used it as a source of potentially exotic $4$-spheres.

In general, a  major difficulty in  Floer-theoretic approaches to proving that some integral slope characterizes a knot $K$ is that  one must first identify all knots with the same knot Floer homology as $K$, and this was out of reach until recently for all but a  handful of knots.  However, Theorem~\ref{thm:main} is made possible by our  recent classification \cite{bs-nonfibered} of all  genus-1 \emph{nearly fibered} knots:

\begin{theorem}[{\cite[Theorem~1.2]{bs-nonfibered}}] \label{thm:nf}
Let $K \subset S^3$ be a genus-1 knot with $\dim_\Q \hfkhat(K,1) = 2$.  Then up to mirroring $K$ must be one of
\begin{equation} \label{eq:nf-plus2t}
5_2, \ 15n_{43522}, \ \Wh^-(T_{2,3},2)
\end{equation}
or
\begin{equation} \label{eq:nf-minus2t}
\Wh^+(T_{2,3},2), \ P(-3,3,2n+1) \ (n\in\Z),
\end{equation}
where the knots in \eqref{eq:nf-plus2t} have Alexander polynomial $\Delta_K(t) = 2t-3+2t^{-1}$ and determinant $|\Delta_K(-1)|=7$, and those in \eqref{eq:nf-minus2t} have Alexander polynomial $\Delta_K(t) = -2t+5-2t^{-1}$ and determinant $|\Delta_K(-1)|=9$.
\end{theorem}

For example, we were able to use this classification to prove in \cite{bs-characterizing} that  all rational slopes besides the positive integers (i.e., not just $0$) are characterizing for $5_2$:

\begin{theorem}[{\cite[Theorem~1.1]{bs-characterizing}}] \label{thm:5_2}
Every $r\in \Q\setminus \Z_{>0}$ is a characterizing slope for $5_2$.
\end{theorem}

We do not expect  anything as strong as Theorem~\ref{thm:5_2} to hold for the knots in Theorem~\ref{thm:main}.  Indeed, Baker and Motegi \cite[Example~4.1]{baker-motegi} proved that $P(-3,3,5)$ is not characterized by any non-zero integer surgeries.  On the other hand, Theorem~\ref{thm:main} gives an affirmative answer to \cite[Question~4.4]{baker-motegi}, which asked whether $0$ might be a characterizing slope for $P(-3,3,5)$.

In this paper we assume some background in Heegaard Floer homology, but the Floer-theoretic techniques we use were all present in \cite{bs-characterizing}; the casual reader may be relieved to know that unlike in \cite{bs-characterizing}, we make no use of the ``mapping cone'' formula for the Heegaard Floer homology of surgeries on a knot.  On the other hand, Floer theoretic invariants cannot distinguish the $0$-surgeries on any of the pretzel knots $P(-3,3,2n+1)$, so we will eventually need to introduce some perturbative invariants defined by Ohtsuki \cite{ohtsuki-perturbative} which can tell them apart.

\subsection*{Organization}

Theorem~\ref{thm:main} is proved in several steps.  In Section~\ref{sec:generalities} we prove some general facts about $0$-surgery on knots of genus one, and then we use these in Section~\ref{sec:det-7} to prove Theorem~\ref{thm:det-7}, stating that $0$-surgery characterizes $15n_{43522}$ and $\Wh^-(T_{2,3},2)$ as well as their mirrors.  In Section~\ref{sec:det-9-1}, we use JSJ decompositions to deal with $\Wh^+(T_{2,3},2)$ and its mirror in Theorem~\ref{thm:wh}.  Then in Section~\ref{sec:det-9-2} we use Ohtsuki's invariants to prove in Theorem~\ref{thm:pretzels} that $0$ is a characterizing slope for each of the pretzel knots $P(-3,3,2n+1)$. We prove as a bonus in Proposition \ref{prop:pretzel-nonzero} that $r$-surgery distinguishes these pretzel knots for any  $r\in\Q$.

\subsection*{Acknowledgments}

We thank Tam Cheetham-West and Alan Reid for some interesting conversations which inspired this work, and in particular for sharing a draft of Tam's article \cite{cheetham-west}.  We also thank the referee for helpful feedback on the initial version of this paper.  JAB was supported by NSF FRG Grant DMS-1952707.

\section{Zero-surgery on genus-one knots} \label{sec:generalities}

We begin by introducing some general results that will let us reduce Theorem~\ref{thm:main} to the case where $J$ is one of the knots listed in Theorem~\ref{thm:nf}.

\begin{proposition} \label{prop:genus-1-general}
Let $K \subset S^3$ be a knot with Seifert genus $1$, and suppose for some other knot $J \subset S^3$ that there is an orientation-preserving homeomorphism
\[ S^3_0(K) \cong S^3_0(J). \]
Then $J$ has genus 1 and the same Alexander polynomial as $K$, and moreover
\[ \dim_\F \hfkhat(K,1) = \dim_\F \hfkhat(J,1) \]
over any field $\F$.
\end{proposition}

\begin{proof}
The manifold $S^3_0(J)$ determines the Alexander polynomial of $J$, because the infinite cyclic covers of both $S^3_0(J)$ and the knot exterior $S^3 \setminus N(J)$ have the same first homology as $\Z[t^{\pm1}]$-modules, so $\Delta_K(t) = \Delta_{J}(t)$.  Gabai \cite{gabai-foliations3} proved that it also determines the Seifert genus $g(J)$, so $g(J) = g(K) = 1$.

We now study the Heegaard Floer homology of various surgeries on $K$, which for the remainder of this proof we will always take with coefficients in a fixed field $\F$.  We recall that there is a smooth concordance invariant $V_0(K) \in \Z$, defined by Rasmussen \cite{rasmussen-thesis}, which can be extracted from the knot Floer complex $\cfkinfty(K)$.  Its precise definition does not matter here, except to note that it appears in computing the Heegaard Floer correction terms of surgeries on $K$, by a formula of Ni and Wu \cite[Proposition~1.6]{ni-wu} which implies
\begin{equation} \label{eq:d-1-surgery}
d(S^3_1(K)) = -2V_0(K)
\end{equation}
as a special case.

The correction terms of the zero-surgery on $K$ satisfy
\begin{equation*}
\begin{aligned}
d_{1/2}(S^3_0(K)) &= \hphantom{-}\tfrac{1}{2} - 2V_0(K) \\
d_{-1/2}(S^3_0(K)) &= -\tfrac{1}{2} + 2V_0(\mirror{K}),
\end{aligned}
\end{equation*}
by \cite[Proposition~4.12]{osz-absolutely} and \eqref{eq:d-1-surgery}.  The same is true for $J$, and these correction terms for $S^3_0(K)$ and $S^3_0(J)$ must agree since $S^3_0(K) \cong S^3_0(J)$, so we have
\begin{equation} \label{eq:v0-k-kp}
V_0(K) = V_0(J).
\end{equation}

Now since $g(K) = 1$ we can apply \cite[Lemma~2.8]{bs-characterizing} to see that $\hfred(S^3_1(K))$ is an $\F[U]$-module with trivial $U$-action, and that
\[ \dim \hfred(S^3_1(K)) = \dim \hfkhat(K,1) - V_0(K). \]
This means that
\[ \hfp(S^3_1(K)) \cong \frac{\F[U,U^{-1}]}{U\cdot \F[U]} \oplus \F^{\dim \hfkhat(K,1) - V_0(K)} \]
as ungraded $\F[U]$-modules, so from the exact triangle
\[ \cdots \to \hfhat(S^3_1(K)) \to \hfp(S^3_1(K)) \xrightarrow{U} \hfp(S^3_1(K)) \to \cdots \]
we deduce that
\[ \dim \hfhat(S^3_1(K)) = 2\left(\dim \hfkhat(K,1) - V_0(K)\right) + 1. \]
Now we apply the surgery exact triangle 
\[ \cdots \to \hfhat(S^3) \to \hfhat(S^3_0(K)) \to \hfhat(S^3_1(K)) \to \cdots \]
to see that
\begin{equation} \label{eq:hfhat-k-0}
\dim \hfhat(S^3_0(K)) = 2\left(\dim \hfkhat(K,1) - V_0(K)\right) + 1 \pm 1.
\end{equation}
The same is true for $J$ since $g(J)=1$ as well, namely
\begin{equation} \label{eq:hfhat-kp-0}
\dim \hfhat(S^3_0(J)) = 2\left(\dim \hfkhat(J,1) - V_0(J)\right) + 1 \pm 1.
\end{equation}
But $\hfhat(S^3_0(K)) \cong \hfhat(S^3_0(J))$ since the two manifolds are the same, so we combine \eqref{eq:hfhat-k-0} and \eqref{eq:hfhat-kp-0} together with \eqref{eq:v0-k-kp} to get
\begin{equation} \label{eq:hfkhat-k-minus-kp}
2\left(\dim \hfkhat(K,1) - \dim \hfkhat(J,1)\right) \in \{-2,0,2\}.
\end{equation}

Now we recall that $\hfkhat(K)$ carries a $\Z$-valued Maslov grading, and that each $\hfkhat(K,i)$ has Euler characteristic equal to the $t^i$-coefficient of $\Delta_K(t)$.  Since $\Delta_K(t) = \Delta_{J}(t)$, this means that
\[ \chi(\hfkhat(K,1)) = \chi(\hfkhat(J,1)), \]
and in particular this implies that
\[ \dim \hfkhat(K,1) \equiv \dim \hfkhat(J,1) \pmod{2}. \]
But then the left side of \eqref{eq:hfkhat-k-minus-kp} is a multiple of $4$, so it must be zero, and thus $\dim\hfkhat(K,1)=\dim\hfkhat(J,1)$ as claimed.
\end{proof}

\begin{remark}The analogue of the $\hfkhat$ claim in Proposition~\ref{prop:genus-1-general} for $g\geq 2$ is that if $S^3_0(K) \cong S^3_0(J)$ then $\hfkhat(K,g) \cong \hfkhat(J,g)$. This has long been known  because in that case \cite[Corollary~4.5]{osz-knot} identifies $\hfkhat(K,g)$ with $\hfp(S^3_0(K),\spinc_{g-1})$ for a certain $\spc$ structure $\spinc_{g-1}$.
\end{remark}

\section{The determinant-7 case} \label{sec:det-7}

Proposition~\ref{prop:genus-1-general} allows us to take care of the knots in Theorem~\ref{thm:nf} with Alexander polynomial $2t-3+2t^{-1}$, using only classical invariants from now on.

\begin{theorem} \label{thm:det-7}
Let $K$ be one of $15n_{43522}$, $\Wh^-(T_{2,3},2)$, or their mirrors.  If $S^3_0(K) \cong S^3_0(J)$ for some knot $J$, then $J$ is isotopic to $K$.
\end{theorem}

\begin{proof}
In each case we have $\Delta_K(t) = 2t - 3 + 2t^{-1}$ and $\dim_\Q \hfkhat(K,1) = 2$.  Thus Proposition~\ref{prop:genus-1-general} says that the same is true of $J$, and then by Theorem~\ref{thm:nf} we know that $J$ must be one of the knots listed in \eqref{eq:nf-plus2t} up to mirroring.  In fact, it cannot be $5_2$ or its mirror, because we know from Theorem~\ref{thm:5_2} that $0$ is a characterizing slope for each of these.

Next, we claim that $J$ cannot be isotopic to the mirror $\mirror{K}$.  Indeed, if this is the case then
\[ S^3_0(K) \cong S^3_0(\mirror{K}) \cong -S^3_0(K), \]
so if $\chi: H_1(S^3_0(K)) \cong \Z \to \Z/2\Z$ is the unique surjection then the Casson--Gordon invariant $\sigma_1(S^3_0(K),\chi)$ (see \cite{casson-gordon}) must be zero.  This invariant is equal to minus the signature of $K$ \cite[Lemma~3.1]{casson-gordon}, so it follows that $\sigma(K) = 0$.  However, this is impossible because $\Delta_K(t)$ has a conjugate pair of simple roots on the unit circle, at
\[ t =  \tfrac{1}{4}(3 \pm i\sqrt{7}), \]
and these are its only roots. Thus the Tristram--Levine signature $\sigma_K(-1) = \sigma(K)$ must be $\pm 2$, giving a contradiction.

It now remains to be shown that if $K$ is $15n_{43522}$ or its mirror, then $J$ cannot be $\Wh^-(T_{2,3},2)$ or its mirror, and vice versa.  In other words, we need to show that
\[ \pm S^3_0(15n_{43522}) \not\cong \pm S^3_0(\Wh^-(T_{2,3},2)), \]
and we do this by checking that they have different fundamental groups.  This can be done in SnapPy \cite{snappy} by counting $6$-fold covers of each:
\begin{verbatim}
In[1]: M = Manifold('15n43522(0,1)')
In[2]: N = Manifold('16n696530(0,1)')
In[3]: len(M.covers(6))
Out[3]: 3
In[4]: len(N.covers(6))
Out[4]: 21
\end{verbatim}
In particular, the fundamental groups of each have different numbers of index-$6$ subgroups, so they cannot be homeomorphic.
\end{proof}

\begin{remark}
Even with Proposition~\ref{prop:genus-1-general}, we will need more than just classical invariants to address the knots in Theorem~\ref{thm:nf} with Alexander polynomial $-2t+5-2t^{-1}$.  For example, if $P$ is one of the pretzel knots $P(-3,3,2n+1)$, then $P$ is slice and so $\sigma(P) = 0$, meaning that the arguments used in Theorem~\ref{thm:det-7} cannot even distinguish the $0$-surgery on $P$ from the $0$-surgery on its mirror.
\end{remark}

\section{The determinant-9 case, part 1} \label{sec:det-9-1}

We now turn to the knots in Theorem~\ref{thm:nf} with Alexander polynomial $-2t+5-2t^{-1}$.  In order to do this, we will first discuss the JSJ decompositions of their $0$-surgeries.

\begin{lemma} \label{lem:p0-jsj}
Let $Y$ be the result of $0$-surgery on $P(-3,3,2n+1)$ for some $n\in\Z$.  Then $Y$ is a graph manifold: it has a single, non-separating JSJ torus, whose complement is Seifert fibered over the annulus.
\end{lemma}

\begin{proof}
We know that $Y$ is toroidal, because if $\Sigma$ is a genus-1 Seifert surface for $P=P(-3,3,2n+1)$ then it extends to a non-separating torus $\hat\Sigma$ after performing $0$-surgery on $P$, and $\hat\Sigma$ is incompressible by \cite[Corollary~8.2]{gabai-foliations3}.  Since $P$ is a Montesinos knot other than a trefoil, Ichihara and Jong \cite{ichihara-jong} proved that $S^3_0(P)$ cannot be toroidal and Seifert fibered, so $Y$ is not Seifert fibered.  On the other hand, if we cut $Y$ open along the torus $\hat\Sigma$ then Cantwell and Conlon \cite[Theorem~1.5]{cantwell-conlon-52} proved that the resulting manifold is the complement of the $(2,4)$-torus link $T_{2,4} \subset S^3$, which is Seifert fibered over the annulus.
\end{proof}

\begin{lemma} \label{lem:w0-jsj}
Let $Y$ be the result of $0$-surgery on $\Wh^+(T_{2,3},2)$.  Then $Y$ is a graph manifold, and its JSJ decomposition consists of two pieces: one piece is the exterior of $T_{2,3}$, and the other is Seifert fibered over a pair of pants.
\end{lemma}

\begin{proof}
Let $W = \Wh^+(T_{2,3},2)$.  We observe that $W$ is a satellite, with companion $C = T_{2,3}$; its pattern $P$ has winding number $0$, hence is not a $0$- or $1$-bridge braid in the solid torus $V = S^1\times D^2$.  This means that $0$-surgery on the pattern $P \subset V$ produces a manifold with incompressible torus boundary, by \cite[Theorem~1.1]{gabai-solid-tori}.  Thus the companion torus $T = \partial N(C)$ in the exterior of $W$ remains incompressible in $Y=S^3_0(W)$.  In particular $T$ is one of the JSJ tori of $S^3_0(W)$, and moreover it separates $S^3_0(W)$ into the union of $S^3 \setminus N(T_{2,3})$ (which is Seifert fibered) and $V_0(P)$.

We claim that $V_0(P)$ is not Seifert fibered.  Indeed, if it were then all but at most one Dehn filling of its boundary would also be Seifert fibered.  But for any $n$ we can realize one of these Dehn fillings by doing $(0,\frac{1}{n}$)-surgery on the Whitehead link, and these are homeomorphic to $0$-surgeries on infinitely many different twist knots.  The only twist knots with a toroidal, Seifert fibered surgery are the trefoils \cite{ichihara-jong}, however, so $V_0(P)$ cannot be Seifert fibered after all.

On the other hand, that the pattern $P$ has a genus-1 Seifert surface $\Sigma$ which lies entirely inside $V$, and which extends to a non-separating, incompressible torus $\hat\Sigma$ in $V_0(P) \subset S^3_0(W)$.  According to \cite[Theorem~7.1]{bs-nonfibered}, if we cut $S^3_0(W)$ open along $\hat\Sigma$ then we are left with the complement of the $(2,4)$-cable of $T_{2,3}$, where the companion torus is the same torus $T$ discussed above.  It follows that cutting $V_0(P)$ along $\hat\Sigma$ produces the complement of a $(2,4)$-torus link in the solid torus, and this is Seifert fibered over a pair of pants.  We conclude that $T$ and $\hat\Sigma$ are the JSJ tori of $S^3_0(W)$, and that $S^3_0(W)$ has the claimed JSJ decomposition.
\end{proof}

Lemmas~\ref{lem:p0-jsj} and \ref{lem:w0-jsj} make it easy to distinguish $0$-surgery on $\Wh^+(T_{2,3},2)$ from the $0$-surgeries on the $P(-3,3,2n+1)$ pretzel knots.

\begin{theorem} \label{thm:wh}
Let $K$ be either $\Wh^+(T_{2,3},2)$ or its mirror.  If $S^3_0(J) \cong S^3_0(K)$ for some knot $J \subset S^3$, then $J$ is isotopic to $K$.
\end{theorem}

\begin{proof}
By Proposition~\ref{prop:genus-1-general}, we see that $J$ has genus 1 and top knot Floer homology
\[ \hfkhat(J,1;\Q) \cong \hfkhat(K,1;\Q) \cong \Q^2, \]
and its Alexander polynomial is $-2t+5-2t^{-1}$.  According to Theorem~\ref{thm:nf}, we therefore know that $J$ is either $K$, its mirror $\mirror{K}$, or some pretzel knot $P(-3,3,2n+1)$.  (We note here that the mirror of $P(-3,3,2n+1)$ is $P(-3,3,-2n-1)$.)

In order to show that $J$ cannot be $\mirror{K}$, we consider the JSJ decompositions of
\[ S^3_0(K) \quad\text{and}\quad S^3_0(\mirror{K}) \cong -S^3_0(K). \]
One of these two manifolds is $S^3_0(\Wh^+(T_{2,3},2))$, and by Lemma~\ref{lem:w0-jsj} its JSJ decomposition consists of two pieces, one of which is the exterior of $T_{2,3}$ and the other of which is not a knot complement.  But then the other manifold decomposes into the exterior of $T_{-2,3}$ and another piece, which is again not a knot complement.  By the uniqueness of the JSJ decomposition, any orientation-preserving homeomorphism $S^3_0(K) \xrightarrow{\cong} -S^3_0(K)$ would have to restrict to an orientation-preserving homeomorphism
\[ S^3 \setminus N(T_{2,3}) \cong S^3 \setminus N(T_{-2,3}), \]
and this is impossible.

Now if $J = P(-3,3,2n+1)$ then Lemma~\ref{lem:p0-jsj} says that the JSJ decomposition of $S^3_0(J)$ consists of a single Seifert fibered piece.  This does not match the decomposition of $S^3_0(K)$, so again we must have $S^3_0(K) \not\cong S^3_0(J)$.  We have now shown that $J$ cannot be either $\bar{K}$ or any of the pretzel knots $P(-3,3,2n+1)$, so $J$ must be isotopic to $K$ after all.
\end{proof}

\section{The determinant-9 case, part 2} \label{sec:det-9-2}

In this section we prove that $0$ is a characterizing slope for each pretzel knot $P(-3,3,2n+1)$.  We begin with the following.

\begin{lemma} \label{lem:these-pretzels-are-making-me-thirsty}
If $S^3_0(J) \cong S^3_0(P(-3,3,2n+1))$ for some $n\in\Z$, then $J$ is isotopic to the pretzel knot $P(-3,3,2m+1)$ for some $m\in\Z$.
\end{lemma}

\begin{proof}
Just as in the proof of Theorem~\ref{thm:wh}, we apply Proposition~\ref{prop:genus-1-general} and Theorem~\ref{thm:nf} to see that if we write $W = \Wh^+(T_{2,3},2)$ then $J$ must be one of
\[ W,\ \mirror{W},\ \text{or}\ P(-3,3,2m+1) \ (m\in\Z). \]
On the other hand, Theorem~\ref{thm:wh} tells us that
\[ S^3_0(W) \not\cong S^3_0(P(-3,3,2n+1)) \quad\text{and}\quad S^3_0(\mirror{W}) \not\cong S^3_0(P(-3,3,2n+1)), \]
so $J$ cannot be $W$ or $\mirror{W}$, hence it must be some $P(-3,3,2m+1)$.
\end{proof}

In order to distinguish the $3$-manifolds $S^3_0(P(-3,3,2n+1))$ for different values of $n$, we use Ohtsuki's perturbative invariants of 3-manifolds $M$ with $b_1(M)=1$ \cite{ohtsuki-perturbative}, which take the form of a power series
\[ \tau(M;c) = \sum_{\ell=0}^\infty \lambda_\ell(M;c)(q-1)^\ell \in \C[[q-1]] \]
that can be evaluated at $c=0$ or at any root $c$ of the Alexander polynomial $\Delta_M(t)$.  Each $\lambda_\ell(M;c)$ is itself an invariant of $M$, and $\lambda_0(M;c)$ is determined by the Alexander polynomial of $M$ \cite[Proposition~5.3]{ohtsuki-perturbative}, so we will compute $\lambda_1(S^3_0(P(-3,3,2n+1)),0)$.

According to the discussion in \cite[\S1]{ohtsuki-perturbative}, we have
\[ \lambda_\ell(S^3_0(K); c) = -\frac{1}{2}\cdot \frac{1+c}{1-c} \left( \mathop{\operatorname{Res}}\limits_{t=c} \frac{(1-t^{-1})^2 P_\ell(t)}{\Delta_K(t)^{2\ell+1}} \right), \]
where the Laurent polynomials $P_\ell(t)$ are the coefficients of the loop expansion
\[ J_n(K;q) = \sum_{\ell=0}^\infty \frac{P_\ell(q^n)}{\Delta_K(q^n)^{2\ell+1}} (q-1)^\ell \]
of the colored Jones polynomial.  We have $P_0(t)=1$ regardless of $K$, and then Ohtsuki \cite[Proposition~6.1]{ohtsuki-cabling} computed that
\begin{equation} \label{eq:p1-formula}
P_1(t) = -(t^{1/2}-t^{-1/2})^2 \cdot \hat\Theta_K(t),
\end{equation}
where the last factor
\[ \hat\Theta_K(t) = \frac{\Theta_K(t,1)}{(t^{1/2}-t^{-1/2})^2} \in \Q[t,t^{-1}] \]
is a specialization of a polynomial called the ``2-loop polynomial'' $\Theta_K(t_1,t_2)$ arising from the Kontsevich integral of $K$.  (We note that the polynomial $J_n(K;q)$ in \cite{ohtsuki-perturbative} is the same as the one denoted $V_n(K;q)$ in \cite{ohtsuki-cabling} -- both are normalized to take the value $1$ when $K$ is the unknot -- and also that \eqref{eq:p1-formula} may differ from the value in \cite{ohtsuki-perturbative} by a sign, but this only changes the invariants $\lambda_1(S^3_0(K);c)$ that we will compute by an overall sign.)

The calculation of these polynomials was described in part by Ohtsuki \cite{ohtsuki-loop}, including a computation of both $\Theta_K(t_1,t_2)$ and $\hat\Theta_K(t)$ when $K$ is a 3-stranded pretzel knot:

\begin{lemma}[{\cite[Example~3.6]{ohtsuki-loop}}] \label{lem:ohtsuki-pretzel}
For the pretzel knot $K = P(p,q,r)$, if we let
\[ d = \frac{pq+qr+rp+1}{4} \]
then the reduced 2-loop polynomial of $K$ is given by 
\[ \hat\Theta_K(t) = \tfrac{1}{16}\big((p+q+r)(4d+1)+pqr\big) \left(-2 - \frac{2d+1}{3}(t-2+t^{-1})\right). \]
\end{lemma}

Applying Lemma~\ref{lem:ohtsuki-pretzel} when $(p,q,r)=(-3,3,2n+1)$, we have $d = -2$ and then
\begin{equation} \label{eq:pretzel-2-loop}
\hat\Theta_{P(-3,3,2n+1)}(t) = -(2n+1)\left(t-4+t^{-1}\right),
\end{equation}
whence for $K = P(-3,3,2n+1)$ we have $\Delta_K(t) = -2t+5-2t^{-1}$ and
\begin{equation} \label{eq:pretzel-p1}
\begin{aligned}
P_1(t) &= -(t-2+t^{-1}) \cdot \hat\Theta_{K}(t) \\
&= (2n+1)(t-2+t^{-1})(t-4+t^{-1}) \\
&= (2n+1)(t^2 - 6t + 10 - 6t^{-1} + t^{-2}) \\
&= (2n+1)\left( \tfrac{1}{4}\Delta_K(t)^2 + \tfrac{1}{2}\Delta_K(t) - \tfrac{3}{4}\right).
\end{aligned}
\end{equation}
The reason for writing it this way is that we can compute $\lambda_1(S^3_0(K),0)$ via the following lemma.

\begin{lemma}[{\cite[Proposition~1.7(2)]{ohtsuki-perturbative}}] \label{lem:residue-shortcut}
Suppose that the Alexander polynomial of $K$ has degree $1$, and write
\begin{align*}
\Delta_K(t) &= b_0 - b_1(t-2+t^{-1}), \\
P_1(t) &= f(t)\Delta_K(t)^3 + a_2 \Delta_K(t)^2 + a_1\Delta_K(t) + a_0
\end{align*}
for some constants $b_0,b_1,a_0,a_1,a_2 \in \Q$ and Laurent polynomial $f(t)$.  Then
\[ \lambda_1(S^3_0(K); 0) = -\frac{d}{2} + \frac{a_2}{2b_1} \]
where $d$ is the constant term of $(t-2+t^{-1})f(t)$.
\end{lemma}

\begin{theorem} \label{thm:pretzels}
Fix an integer $n \in \Z$.  If $S^3_0(P(-3,3,2n+1)) \cong S^3_0(K)$ for some knot $K \in S^3$, then $K$ is isotopic to $P(-3,3,2n+1)$.
\end{theorem}

\begin{proof}
Lemma~\ref{lem:these-pretzels-are-making-me-thirsty} guarantees that $K$ is $P(-3,3,2m+1)$ for some $m\in\Z$.  We use Lemma~\ref{lem:residue-shortcut} for $P(-3,3,2n+1)$: we have $(b_0,b_1) = (1,2)$, and \eqref{eq:pretzel-p1} tells us that
\[ (f(t),a_2,a_1,a_0) = \left(0, \frac{2n+1}{4}, \frac{2n+1}{2}, -\frac{3(2n+1)}{4}\right). \]
The constant term of $(t-2+t^{-1})f(t) = 0$ is $d=0$, so we end up with
\[ \lambda_1(S^3_0(P(-3,3,2n+1)); 0) = \frac{a_2}{2b_1} = \frac{2n+1}{16}. \]
But then an identical calculation says that
\[ \lambda_1(S^3_0(P(-3,3,2m+1)); 0) = \frac{2m+1}{16}, \]
and since these two invariants agree, we must have $m=n$.
\end{proof}

In fact, we can distinguish surgeries of any slope on these pretzel knots.

\begin{proposition} \label{prop:pretzel-nonzero}
If $r \in \Q$ is non-zero and $m$ and $n$ are distinct integers, then \[ S^3_r(P(-3,3,2m+1)) \not\cong S^3_r(P(-3,3,2n+1)). \]\end{proposition}

\begin{proof}
This uses an LMO invariant obstruction due to Ito \cite{ito-lmo}, just as in \cite[\S7]{bs-characterizing}: both knots have the same Conway polynomial $\nabla_K(z) = 1-2z^2$, with the same $z^4$-coefficient
\[ a_4(P(-3,3,2m+1)) = a_4(P(-3,3,2n+1)) = 0. \]
Thus if their $r$-surgeries are homeomorphic, then by \cite[Corollary~1.3(iv)]{ito-lmo} these knots must have the same finite type invariants
\[ v_3(P(-3,3,2m+1)) = v_3(P(-3,3,2n+1)). \]
But Ohtsuki \cite[Proposition~1.1]{ohtsuki-loop} proved that $v_3(K) = \frac{1}{2}\hat\Theta_K(1)$, and so \eqref{eq:pretzel-2-loop} says that \[ v_3(P(-3,3,2n+1))=2n+1, \] hence these pretzel knots have different $v_3$ invariants unless $m=n$.  (We note that Ohtsuki's normalization of $v_3$ differs from Ito's by a scalar, but this does not affect the argument.)
\end{proof}

We remark that Ito's obstruction cannot be used to prove Theorem~\ref{thm:pretzels}, however, because it only applies to non-zero surgeries.  Moreover, Proposition~\ref{prop:pretzel-nonzero} does not prove that non-zero slopes are characterizing for these pretzel knots, because for example the Heegaard Floer homology of $S^3_r(K) \cong S^3_r(P(-3,3,2n+1))$ may not suffice to determine $\hfkhat(K)$ when $r \neq 0$.

\bibliographystyle{myalpha}
\bibliography{References}

\begin{thebibliography}{AJOT13}

\bibitem[AJOT13]{AJOT}
T.~Abe, I.~D. Jong, Y.~Omae, and M.~Takeuchi.
\newblock Annulus twist and diffeomorphic 4-manifolds.
\newblock {\em Math. Proc. Cambridge Philos. Soc.}, 155(2):219--235, 2013.

\bibitem[BM18]{baker-motegi}
K.~L. Baker and K.~Motegi.
\newblock Noncharacterizing slopes for hyperbolic knots.
\newblock {\em Algebr. Geom. Topol.}, 18(3):1461--1480, 2018.

\bibitem[Bra80]{brakes}
W.~R. Brakes.
\newblock Manifolds with multiple knot-surgery descriptions.
\newblock {\em Math. Proc. Cambridge Philos. Soc.}, 87(3):443--448, 1980.

\bibitem[BS22a]{bs-characterizing}
J.~A. Baldwin and S.~Sivek.
\newblock Characterizing slopes for $5_2$.
\newblock arXiv:2209.09805, 2022.

\bibitem[BS22b]{bs-nonfibered}
J.~A. Baldwin and S.~Sivek.
\newblock Floer homology and non-fibered knot detection.
\newblock arXiv:2208.03307, 2022.

\bibitem[CC93]{cantwell-conlon-52}
J.~Cantwell and L.~Conlon.
\newblock Foliations of {$E(5_2)$} and related knot complements.
\newblock {\em Proc. Amer. Math. Soc.}, 118(3):953--962, 1993.

\bibitem[CDGW]{snappy}
M.~Culler, N.~M. Dunfield, M.~Goerner, and J.~R. Weeks.
\newblock Snap{P}y, a computer program for studying the geometry and topology
  of $3$-manifolds.
\newblock Available at \url{http://snappy.computop.org} (04/11/2022).

\bibitem[CG78]{casson-gordon}
A.~J. Casson and C.~M. Gordon.
\newblock On slice knots in dimension three.
\newblock In {\em Algebraic and geometric topology ({P}roc. {S}ympos. {P}ure
  {M}ath., {S}tanford {U}niv., {S}tanford, {C}alif., 1976), {P}art 2}, Proc.
  Sympos. Pure Math., XXXII, pages 39--53. Amer. Math. Soc., Providence, R.I.,
  1978.

\bibitem[CW23]{cheetham-west}
T.~Cheetham-West.
\newblock Distinguishing some genus one knots using finite quotients.
\newblock {\em J. Knot Theory Ramifications}, 32(5):Paper No. 2350035, 7, 2023.

\bibitem[Gab87]{gabai-foliations3}
D.~Gabai.
\newblock Foliations and the topology of {$3$}-manifolds. {III}.
\newblock {\em J. Differential Geom.}, 26(3):479--536, 1987.

\bibitem[Gab89]{gabai-solid-tori}
D.~Gabai.
\newblock Surgery on knots in solid tori.
\newblock {\em Topology}, 28(1):1--6, 1989.

\bibitem[IJ10]{ichihara-jong}
K.~Ichihara and I.~D. Jong.
\newblock Toroidal {S}eifert fibered surgeries on {M}ontesinos knots.
\newblock {\em Comm. Anal. Geom.}, 18(3):579--600, 2010.

\bibitem[Ito20]{ito-lmo}
T.~Ito.
\newblock On {LMO} invariant constraints for cosmetic surgery and other surgery
  problems for knots in {$S^3$}.
\newblock {\em Comm. Anal. Geom.}, 28(2):321--349, 2020.

\bibitem[MP18]{miller-piccirillo}
A.~N. Miller and L.~Piccirillo.
\newblock Knot traces and concordance.
\newblock {\em J. Topol.}, 11(1):201--220, 2018.

\bibitem[MP21]{manolescu-piccirillo}
C.~Manolescu and L.~Piccirillo.
\newblock From zero surgeries to candidates for exotic definite four-manifolds.
\newblock arXiv:2102.04391, 2021.

\bibitem[NW15]{ni-wu}
Y.~Ni and Z.~Wu.
\newblock Cosmetic surgeries on knots in {$S^3$}.
\newblock {\em J. Reine Angew. Math.}, 706:1--17, 2015.

\bibitem[Oht04]{ohtsuki-cabling}
T.~Ohtsuki.
\newblock A cabling formula for the 2-loop polynomial of knots.
\newblock {\em Publ. Res. Inst. Math. Sci.}, 40(3):949--971, 2004.

\bibitem[Oht07]{ohtsuki-loop}
T.~Ohtsuki.
\newblock On the 2-loop polynomial of knots.
\newblock {\em Geom. Topol.}, 11:1357--1475, 2007.

\bibitem[Oht10]{ohtsuki-perturbative}
T.~Ohtsuki.
\newblock Perturbative invariants of 3-manifolds with the first {B}etti number
  1.
\newblock {\em Geom. Topol.}, 14(4):1993--2045, 2010.

\bibitem[OS03]{osz-absolutely}
P.~Ozsv\'{a}th and Z.~Szab\'{o}.
\newblock Absolutely graded {F}loer homologies and intersection forms for
  four-manifolds with boundary.
\newblock {\em Adv. Math.}, 173(2):179--261, 2003.

\bibitem[OS04]{osz-knot}
P.~Ozsv\'{a}th and Z.~Szab\'{o}.
\newblock Holomorphic disks and knot invariants.
\newblock {\em Adv. Math.}, 186(1):58--116, 2004.

\bibitem[Oso06]{osoinach}
J.~K. Osoinach, Jr.
\newblock Manifolds obtained by surgery on an infinite number of knots in
  {$S^3$}.
\newblock {\em Topology}, 45(4):725--733, 2006.

\bibitem[Pic19]{piccirillo-shake}
L.~Piccirillo.
\newblock Shake genus and slice genus.
\newblock {\em Geom. Topol.}, 23(5):2665--2684, 2019.

\bibitem[Pic20]{piccirillo-conway}
L.~Piccirillo.
\newblock The {C}onway knot is not slice.
\newblock {\em Ann. of Math. (2)}, 191(2):581--591, 2020.

\bibitem[Ras03]{rasmussen-thesis}
J.~A. Rasmussen.
\newblock {\em Floer homology and knot complements}.
\newblock ProQuest LLC, Ann Arbor, MI, 2003.
\newblock Thesis (Ph.D.)--Harvard University.

\bibitem[Yas15]{yasui-corks}
K.~Yasui.
\newblock Corks, exotic 4-manifolds and knot concordance.
\newblock arXiv:1505.02551, 2015.

\end{thebibliography}

\end{document}